\documentclass[12pt]{amsart}

\usepackage{etex}
\usepackage[english]{babel}
\usepackage[cp1250]{inputenc}
\usepackage{amssymb,amsthm,amsfonts,amsmath}
\usepackage{mathrsfs}
\usepackage{verbatim}
\usepackage{ulem}
\usepackage{mathtools}

\usepackage{url}
\usepackage{amsopn}
\usepackage{graphicx}
\usepackage[usenames, dvipsnames]{xcolor}
\DeclareGraphicsExtensions{.pdf,.png,.jpg}

\usepackage[shortlabels]{enumitem}
\usepackage{titletoc}
\definecolor{darkblue}{rgb}{0.0, 0.0, 0.55}
\usepackage[colorlinks,linkcolor=BrickRed,citecolor=OliveGreen,urlcolor=darkblue,hypertexnames=true,backref=true]{hyperref}
\usepackage[a4paper,margin=2.5cm]{geometry}
\numberwithin{equation}{section}

\newcommand{\gengam}[2]{ {\Gamma_{#1,#2} (\cD_A) } }

\newcommand{\C}{\mathbb C}

\newcommand{\RR}{\mathbb R}
\newcommand{\FF}{\mathbb F}

\newcommand{\CC}{\mathbb C}

\newcommand{\cD}{\mathcal D}

\newcommand{\cS}{\mathcal S}

\def\dropAB{ { {\mathrm{proj}_x\hspace{0.05cm}  \cD_{(A,B)} } } }

\def\hbeta{{\hat{\beta}}}

\def\hpsi{{\hat{\psi}}}
\def\hY{{\hat{Y}}}

\newtheorem{theorem}{Theorem}[section]
\newtheorem{corollary}[theorem]{Corollary}

\newtheorem{proposition}[theorem]{Proposition}

\theoremstyle{definition}
\newtheorem{definition}[theorem]{Definition}
\newtheorem{remark}[theorem]{Remark}
\newtheorem{example}[theorem]{Example}

\title[Extreme points of matrix convex sets and their spanning properties]{Extreme points of matrix convex sets and their spanning properties}

\author{Eric Evert}
\author{Benjamin Passer}
\thanks{The views expressed in this document are those of the authors and do not reflect the official policy or position of the U.S. Department of Defense or the U.S. Government.}
\author{Tea \v{S}trekelj}

\newcommand{\Addresses}{{
    \bigskip
    \footnotesize

    \noindent\textsc{Department of Computer Science, Northwestern University, Evanston, IL, United States}\par\nopagebreak \textit{E-mail address}: \texttt{eric.evert@northwestern.edu}

    \medskip \medskip
    
    \noindent \textsc{Department of Mathematics, United States Naval Academy, Annapolis, MD, United States}\par\nopagebreak \textit{E-mail address}: \texttt{passer@usna.edu}

    \medskip \medskip

    \noindent \textsc{Institute of Mathematics, Physics and Mechanics, Ljubljana, Slovenia}\par\nopagebreak \textit{E-mail address}: \texttt{tea.strekelj@fmf.uni-lj.si}
  }}

\begin{document}

\begin{abstract}
This expository article gives a survey of matrix convex sets, a natural generalization of convex sets to the noncommutative (dimension-free) setting, with a focus on their extreme points. Mirroring the classical setting, extreme points play an important role in matrix convexity, and a natural question is, ``are matrix convex sets the (closed) matrix convex hull of their extreme points?" That is, does a Krein-Milman theorem hold in this setting? This question requires some care, as there are several notions of extreme points for matrix convex sets. Three of the most prevalent notions are matrix extreme points, matrix exposed points, and free extreme points. For each of these types of extreme points, we examine strengths and shortcomings in terms of a Krein-Milman theorem. Of particular note is the fact that these extreme points are all finite-dimensional in nature. As such, a large amount of our discussion is about free spectrahedra, which are matrix convex sets determined by a linear matrix inequality.

\end{abstract}

\maketitle

 \section{Introduction}

	An important part of the classical theory of convexity is the study of distinguished points of the relative boundary of a (compact) convex set $C$, as these points capture many of the relevant properties of $C$. The extreme points ext($C$) of a convex set $C$ are those points $c \in C$ that cannot be expressed as a nontrivial convex combination of the elements of $C.$ An equivalent geometric description is that the extreme points are not interior points of any line segment lying entirely in $C.$ A cornerstone of the theory of convexity is the Krein-Milman theorem \cite[Section III.4]{Ba}, which states that a compact convex set $C$ is the closed convex hull of its extreme points, so in this case the extreme points generate $C.$\looseness=-1
	
	The exposed points of a convex set $C$ form a subset of the extreme points that is also of interest. Exposed points are the points of $C$ that can be weakly separated from $C$ by a supporting affine hyperplane. For polyhedra in $\RR^n,$ the exposed and extreme points coincide, and the same claim holds for spectrahedra in $\RR^n$ by \cite{RG}. However, in general, exposed points form a proper subset of the extreme boundary. If $C$ is a compact convex set in a normed vector space, the exposed points are dense in the extreme points by the Straszewicz theorem (see \cite[Section II.2]{Ba} and \cite{K}). As a corollary, a compact convex set in a normed vector space is the closed convex hull of its exposed points.
 In this article we review the theory on the extreme boundary of noncommutative convex sets, i.e., matrix convex sets.
	
	\subsection{Matrix convex sets}
	For a  locally convex space $V$, let $M_{m,n}(V)$ denote the space of $m \times n$ matrices over $V.$ For convenience, also let $M_n(V) = M_{n,n}(V)$ and $M(V) = \cup_n M_n(V).$ We will primarily be concerned with the case that $V = \mathbb{F}^g$ where $\mathbb{F} \in \{\mathbb{R},\mathbb{C}\},$ in which case the space $M_n(\mathbb{F}^g)$ is canonically identified with $M_n(\mathbb{F})^g,$ the set of $g$-tuples of $n \times n$ matrices.
 If $V = \mathbb{F},$ denote by $SM_n(\mathbb{F}) \subset M_n(\mathbb{F})$ the set of self-adjoint matrices, let $SM(V) = \cup_n SM_n(V)$,  and write $I_n \in M_n(\mathbb{F})$ for the identity matrix. All the matrix spaces $M_{m,n}(V)$ over $V$ are endowed with the product topology.
 A graded family $\textbf{$S$}= (S_n)_{n \in \mathbb{N}} \subseteq (M_n(V))_{n \in \mathbb{N}}$ is closed (compact) if it is levelwise closed (compact), i.e., for each $n,$ the set $S_n$ is closed (compact). We say that two elements $A,B \in M_n (V)$ are unitarily equivalent, denoted $A\sim_u B$, if there exists a unitary $U \in M_n(\mathbb{F})$ such that $U^* A U = B$. Finally, we say $A \in M(V)$ is reducible if there exist $B \in M(V)$ and $C \in M(V)$ such that $A \sim_u B \oplus C$, and we say $A$ is irreducible if it is not reducible. 

\begin{definition} \cite{Wit}
	Suppose that for each $n \in \mathbb{N}$, the set $K(n)$ is a subset of $M_n(V)$, and denote by \textbf{$K$} the graded set $(K(n))_{n \in \mathbb{N}}$.

 \begin{enumerate}[(a)]
	\item Let $X^1,\ldots, X^k \in \textbf{$K$}$ with $X^i \in K(n_i)$. An expression of the form
	\begin{equation}\label{eq-11}
		\sum_{i=1}^k \gamma_i^\ast X^i \gamma_i,
	\end{equation}
	where $\gamma_i \in M_{n_i,n}(\mathbb{F})$ are matrices with $\sum_{i=1}^k \gamma_i^\ast \gamma_i = I_n$, is a \textbf{matrix convex combination} of the  points  $X^1,\ldots, X^k$. 
	\item We call \textbf{$K$} a \textbf{matrix convex set} in $V$ if it is closed under matrix convex combinations.
 \end{enumerate}

\end{definition}

Equivalently, a graded set \textbf{$K$} is matrix convex if and only if it is closed under formation of direct sums and conjugations by isometries. If $K$ is matrix convex and $0 \in K(1),$ then \textbf{$K$} is also closed under conjugations by arbitrary contractions (see, e.g., \cite[Lemma 2.3]{HKM16}).
Note that for a matrix convex set \textbf{$K$}, each $K(n)$ is a classical convex set.

Matrix convex sets appear naturally in both operator theoretic and matrix theoretic problems. The concept of the  numerical range of a matrix naturally extends to the matrix range of a tuple of bounded operators, which is a closed and bounded matrix convex set over Euclidean space. Matrix ranges (as well as operator systems) are briefly examined in section \ref{subsec:opersys}. Further, just as a polyhedron is a convex set determined by a finite collection of linear inequalities, a free spectrahedron is a matrix convex set that is determined by a finite collection of linear matrix inequalities, or equivalently, a single linear matrix inequality. Free spectrahedra are examined in section \ref{sec:freespecsec}, including both parallels with and departures from the classical case. 

For any graded set $\textbf{$S$} = (S_n)_{n \in \mathbb{N}}$ with $S_n \subseteq M_n(V),$ the 
set of all matrix convex combinations of the elements of $S$ is called the \textbf{matrix convex hull} of $S$ and is denoted by mconv($S$). The matrix convex hull of $S$ is also the intersection of all matrix convex sets containing $S.$ Its closure is denoted by $\overline{\text{mconv}}(\textbf{$S$})$.

\begin{definition}
  Let \textbf{$K$} and \textbf{$L$} be matrix convex sets over spaces $V$ and $W$, respectively. A \textbf{matrix affine map} is a continuous linear map $\Phi : V \to W$ that satisfies $\Phi_r(K(r)) \subseteq L(r)$ for all $r \in \mathbb{N}$ and
	$$
	\Phi_r\bigg(\sum_{i=1}^k \gamma_i^\ast X^i \gamma_i\bigg) =
	\sum_{i=1}^k\gamma_i^\ast \, \Phi_{r_i}(X^i)  \gamma_i
	$$
	 for all $k$-tuples $(X^i)_{i=1}^k$ and $(\gamma_i)_{i=1}^k$ such that $X^i \in K(r_i)$ and $\gamma_i \in M_{r_i,r}(\mathbb{F})$ for $i = 1,\ldots, k$ with the property $\sum_{i=1}^k \gamma_i^\ast \gamma_i = I_r$. Here for any positive integer $r$ and $B = (B_{i,j}) \in M_r(V)$ we denote by 
	$$\Phi_r(B)  = \big(\Phi(B_{i,j})\big)$$ 
	the $r$th \textbf{ampliation} of $\Phi.$ We call $\Phi$ a \textbf{matrix affine homeomorphism} if each map $\Phi_r$ is a homeomorphism.
\end{definition}

A standard example of a matrix affine map between matrix convex sets \textbf{$K$} and \textbf{$L$} over spaces $V$ and $W$ respectively is the family of ampliations $(\Phi_r)_r$ of any linear map $ \Phi_1: V \to W.$

\section{Matrix extreme points}
Matrix convex sets have multiple classes of extreme points. The earliest class, matrix extreme points, was introduced by Webster and Winkler in \cite{WW}. However, we note that $C^\ast$-extreme points, which are the extreme points of single-level noncommutative convex sets (called $C^\ast$-convex sets), were defined earlier by Farenick and Morenz in \cite{FM}.
\begin{definition}\label{mext}
	Let \textbf{$K$} be a matrix convex set.
	\begin{enumerate}[(a)]
	\item A matrix convex combination \eqref{eq-11} is \textbf{proper} if all of the matrices $\gamma_i$ are onto (i.e., they correspond to surjective linear transformations).	
	\item A point $X \in K(n)$ is \textbf{matrix extreme} if from any expression of $X$ as a proper matrix convex combination of elements $X^i \in K(n_i)$, it follows that $n_i = n$ and each of the $X^i$ is unitarily equivalent to $X$.
 \end{enumerate}

\end{definition}

Note in particular that if $X \in K(n)$ is written as a proper matrix convex combination of points of $K$, then those points must be chosen from levels $1$ through $n$. So, determining whether $X \in K(n)$ is a matrix extreme point of $K$ only requires knowledge of $K(n)$ itself, not any higher levels. Denote the set of matrix extreme points of $K$ by mext($K$).

\begin{remark}
    The appearance of unitary equivalence in the definition of a matrix extreme point (in contrast to equality in the classical theory) is natural and in fact needed. Demanding equality instead of unitary equivalence would imply that a matrix extreme point has a trivial unitary orbit, which is far too restrictive. Indeed, any unitary conjugation of a point is a matrix convex combination of that point, so the unitary orbit can always be obtained by matrix convex combinations. 
\end{remark}

Since for $n=1,$ any matrix convex combination \eqref{eq-11} reduces to a scalar convex combination, the matrix extreme points of $K(1)$ are precisely the classical extreme points. In general, the matrix extreme points of $K(n)$ are extreme in the classical sense, but there may be extreme points that are not matrix extreme. Note that if $X$ is matrix extreme, then by writing out a classical convex combination $X = tY + (1-t)Z$ of points $Y$ and $Z$ as\looseness=-1
$$
X =  (\sqrt{t}I)^* Y \sqrt{t}I + (\sqrt{1-t}I)^* Z \sqrt{1-t}I,
$$
one can only deduce that $Y$ and $Z$ are unitarily equivalent (but not necessarily equal) to $X.$
A complete argument for the case of a compact matrix convex set is given by \protect{\cite[Corollary 3.6]{WW}}, which relies on the duality between matrix convex sets and operator systems. However, the claim that matrix extreme points are extreme also holds for non-compact sets; it is for example an easy corollary of the characterization \protect{\cite[Proposition 4.6]{EHKM}} (see also \cite{HL}).

\subsection{The Webster-Winkler Krein-Milman theorem for matrix convex sets}
Since the matrix extreme points of $K(1)$ coincide with the classical extreme points of $K(1)$, we know by the Krein-Milman theorem \cite[Theorem III.4.1]{Ba} that if $K(1)$ is compact, then it has extreme points. In this case, it follows that mext($K$) is also nonempty. The next theorem, which is a matricial Krein-Milman theorem by Webster-Winkler \cite{WW}, states that a compact matrix convex set has  ``enough'' matrix extreme points.

\begin{theorem}\label{th: km}
    Let $K$ be a compact matrix convex set in a locally convex space $V$. Then $\text{\textnormal{mext}}(K) \neq \emptyset$ and
		$$
		K = \overline{\text{\textnormal{mconv}}}(\text{\textnormal{mext}}(K)).
		$$
\end{theorem}
The original idea of the proof by Webster and Winkler was to associate to a matrix convex set a family of classical convex sets and reduce the proof of the matricial Krein-Milman theorem to the classical Krein-Milman theorem. In 2019, Hartz and Lupini \cite{HL} proposed a slight adaptation, which we now present, as it arguably eases the original proof. We associate to a matrix convex set $\textbf{$K$} = (K(r))_{r \in \mathbb{N}}$ in the space $V$ a family of convex sets $\{\Gamma_n(\textbf{$K$})\}_{n \in \mathbb{N}}$ given by
\begin{equation}\label{eq311}
	\Gamma_n(\textbf{$K$}) = \{ (\gamma^\ast \gamma, \gamma^\ast X \gamma) \ | \ \gamma \in M_{k, n}(\mathbb{F}), \text{tr}(\gamma^\ast \gamma) = 1, k \in \mathbb{N}, X \in K(k) \} \subseteq M_n(\mathbb{F}) \times M_n(V).
\end{equation}
To see that $\Gamma_n(\textbf{$K$})$ is indeed convex, note that 
$$
t \gamma^\ast X \gamma + (1-t)\delta^\ast Y \delta = 
\bigg[t^{1/2}\gamma^\ast \ \  (1-t)^{1/2}\delta^\ast\bigg]
\begin{bmatrix}
	X & 0\\
	0 & Y \\
\end{bmatrix}
\begin{bmatrix}
	t^{1/2}\gamma\\
	(1-t)^{1/2}\delta\\
\end{bmatrix}
$$
for elements $(\gamma^\ast \gamma, \gamma^\ast X \gamma)$ and $(\delta^\ast \delta, \delta^\ast Y \delta)$ of $\Gamma_n(\textbf{$K$}),$ where $X\in K(r),$ $Y \in K(s)$, $\gamma \in M_{r, n}(\mathbb{F})$ and $\delta \in M_{s, n}(\mathbb{F})$ are matrices satisfying $\text{tr}(\gamma^\ast \gamma) = \text{tr}(\delta^\ast \delta) = 1,$ and $t \in [0, 1]$ is arbitrary.
Since \textbf{$K$} is closed under direct sums
and
$$
\text{tr}\bigg(
\begin{bmatrix}
	t^{1/2}\gamma\\
	(1-t)^{1/2}\delta\\
\end{bmatrix}^\ast
\begin{bmatrix}
	t^{1/2}\gamma\\
	(1-t)^{1/2}\delta\\
\end{bmatrix}
\bigg)
=
t\,\text{tr}(\gamma^\ast \gamma) + (1-t)\,\text{tr}(\delta^\ast \delta) = 1,
$$
the convex combination $t \big(\gamma^\ast \gamma, \gamma^\ast X \gamma\big) + (1-t)\big(\delta^\ast \delta, \delta^\ast Y \delta\big)$ lies in $\Gamma_n(\textbf{$K$}).$

We remark why for any element $(\gamma^\ast \gamma, \gamma^\ast X \gamma)$ from $\Gamma_n(\textbf{$K$})$, where $X$ is in $K(r)$, the matrix $\gamma \in M_{r, n}(\mathbb{F})$ can be assumed surjective. If $\gamma \in M_{r, n}(\mathbb{F})$ is a matrix of rank $s \in \mathbb{N}$ satisfying tr$(\gamma^\ast \gamma) = 1$  and $\xi \in M_{r, s}(\mathbb{F})$ is an isometry from $\mathbb{F}^s$ to the range of $\gamma,$ then
\begin{equation}\label{izo}
	\gamma^\ast X \gamma = (\xi^* \gamma)^\ast (\xi^\ast X \xi) (\xi^\ast \gamma).
\end{equation}
Note that $\xi^* X \xi$ lies in $K(s)$ and the matrix $\xi^* \gamma$ is surjective with tr$((\xi^* \gamma)^\ast (\xi^* \gamma)) = 1$.

The key to finish the proof is the following proposition \cite[Proposition 2.14]{HL}, which gives the connection between matrix extreme points of $K$ and classical extreme points of $\Gamma_n(K).$

\begin{proposition} \label{pr: gama ext}
    Let $\textbf{$K$} = (K(n))_{n \in \mathbb{N}}$ be a matrix convex set in a vector space $V$. 
    Let $X \in K(k)$ and let $\gamma \in M_{k, n}(\mathbb{F})$ be a surjective matrix with $\text{tr}(\gamma^\ast \gamma) = 1.$ Then $(\gamma^\ast \gamma, \gamma^\ast X \gamma)$ is an extreme point of $\Gamma_n(K)$ if and only if $X$ is a matrix extreme point of $K.$
\end{proposition}
Using the above proposition, the claim that the matrix extreme points span $K$ reduces to the fact that the classical extreme points span $\Gamma_n(K).$
For more detail, we refer the reader to \cite{HL} and \cite{WW}.

\section{Matrix exposed points}
In this section, we investigate matrix exposed points of matrix convex sets, in the case $\mathbb{F} = \mathbb{C}$. The notion of a matrix exposed point originated in \cite{Kr}, and it was generalized to arbitrary infinite-dimensional vector spaces in \cite{KS22}.
The definition of a matrix exposed point aims to capture properties analogous  to those of a classical exposed point. Throughout this section we assume that $K$ is a matrix convex set in a dual  space $V,$ i.e., $V$ is the dual of some vector space $V^\prime.$ We endow $V$ with the weak$^*$ topology, and all linear maps that we discuss are assumed continuous with respect to this topology.


\begin{definition}\label{def511}
	Let $\textbf{$K$} = (K(n))_{n \in \mathbb{N}}$ be a matrix convex set in a dual vector space $V$. An element $X \in K(n)$ is called a \textbf{matrix exposed point} of \textbf{$K$} if there exists a continuous linear map $\Phi : V \to M_n(\C)$ and a self-adjoint matrix $\alpha \in SM_n(\C)$ such that the following conditions hold:\looseness=-1
	\begin{enumerate}[(a)]
		\item \label{p1} for all positive integers $r$ and all $Y \in K(r)$, we have $\Phi_r(Y) \preceq \alpha \otimes I_r;$ 
		
		\item \label{p2} $\{ Y \in K(n) \ | \  \alpha \otimes I_n - \Phi_n(Y) \succeq 0 \text{ singular}\} = \{U^\ast X U \ | \ U \in M_n(\C) \text{ unitary}\}.$
	\end{enumerate}
	We say that such a pair $(\Phi, \alpha)$  \textbf{matricially exposes} the point $X$ and denote the set of all matrix exposed points of \textbf{$K$} by mexp($\textbf{$K$}$).
\end{definition}

\begin{remark}
Condition (b) of Definition \ref{def511}  implies a weak type of separation of the exposed point from the set.  Indeed,  if a linear map $\Phi$, matrix $\alpha \in SM_n(\C)$, and  $X \in K(n)$ are as in Definition \ref{def511}, then we have for any unitary matrix $U \in M_n(\C),$
	\begin{align*}
		\alpha \otimes I_n - \Phi_n(U^\ast X U) &= (I_n \otimes U^\ast) \big(  \alpha \otimes I_n - \Phi_n(X)\big) (I_n \otimes U). 
	\end{align*}
	From here it follows that the matrix $\alpha \otimes I_n - \Phi_n(X)$ is singular if and only if the matrix $\alpha \otimes I_n - \Phi_n(U^\ast X U)$ is singular. Condition (b) of Definition \ref{def511} demands that the only points $Y$ from $K(n)$ for which 
the expression $\alpha \otimes I_n - \Phi_n(Y)$ is singular  are the points from the unitary orbit of $X.$ Hence, if a point $X$ is matrix exposed, then so is any point from its unitary orbit (being exposed by the same pair $(\Phi, \alpha)$ as $X$). 
\end{remark}

The next theorem (see \cite[Poposition 3.5, Theorem 3.8]{KS22} and \cite[Proposition 6.19, Theorem 6.21]{Kr}) highlights the interplay between matrix extreme, matrix exposed, and classical exposed points. 

\begin{theorem}
    Let $\textbf{$K$} = (K(n))_{n \in \mathbb{N}}$ be a matrix convex set in a vector space $V.$ Then the following hold:
    \begin{enumerate}[(a)]        
    \item Every matrix exposed point $X$ in $K(n)$ is a classical exposed point of $K(n)$.
    \item Every matrix exposed point is matrix extreme.
    \item A matrix extreme point is matrix exposed if and only if it is also an exposed point in the classical sense.
    \end{enumerate}
\end{theorem}

\begin{proof}[Main ideas of the proof]
    To prove (a) we want to produce from the given pair $(\Phi, \alpha)$ that matricially exposes $X$ a classical exposing pair $(\varphi, a),$ where $\varphi: M_n(V) \to \CC$ is a continuous functional and $a \in \RR$ with $\varphi(X) = a$ and $\varphi(Y) < a$ for all $Y \in K(n) \backslash \{X\}.$ An immediate idea is to choose a vector $x = \sum_{j=1}^n x_j \otimes e_j \in \mathbb{C}^n \otimes \mathbb{C}^n$ in the kernel of $\alpha \otimes I_n - \Phi_n(X)$ and define
\begin{equation}\label{eq: exp fn}
    \varphi(Y) = x^\ast \Phi_n(Y) x, \quad a = x^\ast (\alpha \otimes I_n) x \in \RR.
\end{equation}
Now clearly $\varphi(X) =a,$ but the vector $x$ also needs to be chosen so that the pair $(\varphi,a)$ has the appropriate separation properties. This choice is justified by \cite[Proposition 3.4]{KS22} and \cite[Corollary 5.25, Remark 5.26]{Kr}, which state that the kernel of $\alpha \otimes I_n - \Phi_n(X)$ is one-dimensional. Moreover, the components $(x_j)_j$ of the vector $x = \sum_{j=1}^n x_j \otimes e_j \in \mathbb{C}^n \otimes \mathbb{C}^n$ that spans the kernel form a basis of $\CC^n.$ It is now easy to check that the pair $(\varphi, a)$ as in \eqref{eq: exp fn} exposes $X$ in $K(n).$ To prove (b), one again exploits the properties of the vector spanning the kernel of $\alpha \otimes I_n - \Phi_n(X).$

For (c), we assume that the matrix extreme point $X$ is exposed, e.g., by the pair $(\varphi, a),$ where $\varphi: M_n(V) \to \CC$ is a continuous functional and $a \in \RR.$ We want to extend the weak separation of $X$ from $K(n)$ given by the pair $(\varphi, a)$ to a matricial weak (Hahn-Banach) separation of $X$ from $K(n)$ given by a pair $(\Phi, \alpha)$ as in Definition \ref{def511}. For that, we rely on a modified version of the Effros-Winkler separation technique \cite{EW}. Indeed, since the point $X$ is matrix extreme, we have by definition that $X$ is not contained in $L:=$ mconv$(K(n)\backslash\{U^\ast X U \ | \ U \in M_n(\mathbb{C}) \text{ is unitary}\})$, and the aim is to strictly (matricially) separate $X$ from $L.$ However, $L$ is a  matrix convex set that is not necessarily closed. Hence, we implicitly use the idea (see \cite{NT}) that one can strictly separate an outer point from a (not necessarily closed) convex set with a functional with values in an ordered extension field of $\RR.$ Then a modified version of the Effros-Winkler technique is applied to produce a candidate for the matricially exposing pair.
Finally,  the fact that the point $X$ is matrix extreme, and hence $X \notin L,$ is used to assert that the candidate indeed has the right exposing properties.
\end{proof}

\subsection{A density and a spanning result}

In the classical setting, if $C$ is a compact convex subset of a normed space $V$, then the exposed points of $C$ not only span $C$, but they are also dense in the extreme points. Below we prove a free analogue of this statement. Key to the proof is the following proposition \cite[Proposition 3.13]{KS22}, which gives an analogue of Proposition \ref{pr: gama ext} for exposed points.

\begin{proposition} \label{tr19}
	Let $\textbf{$K$} = (K(m))_{m \in \mathbb{N}}$ be a matrix convex set and $X \in K(r)$. 
 \begin{enumerate}[(a)]
     \item Let $\gamma \in M_{r, n}(\C)$ be a surjective matrix with $\text{\textnormal{tr}}(\gamma^\ast \gamma)=1$ such that the point
		$(\gamma^\ast \gamma, \gamma^\ast X \gamma)$ is exposed in $\Gamma_n(\textbf{$K$}).$ Then $X$ is a matrix exposed point of \textbf{$K$}.
  \item If $X$ is matrix exposed in \textbf{$K$}, then for any invertible $\gamma \in M_r(\C)$ with $\text{\textnormal{tr}}(\gamma^\ast \gamma)=1,$ the point
		$(\gamma^\ast \gamma, \gamma^\ast X \gamma)$ is exposed in $\Gamma_r(\textbf{$K$}).$
	\end{enumerate}
\end{proposition}

In the classical setting, Milman's converse of the Krein-Milman theorem (see \cite{Ba}) states that if $E$ is any subset of a compact convex set $K$ with the property that $\overline{\text{conv}}(E) = K$, then $\overline{E}$ includes the extreme points of $K$. When applied to the set of exposed points, this converse implies that the exposed points are dense in the extreme points. In the matrix convex setting, because the matricial Krein-Milman theorem does not have a full converse, an analogous result for matrix exposed points cannot be as easily deduced as it is in the classical setting. The following result shows that, nonetheless, the matrix exposed points are dense in the matrix extreme points.

\begin{theorem}\label{th316}
	Let \textbf{$K$} be a compact matrix convex set in a normed vector space $V.$ Then the matrix exposed points of \textbf{$K$} are dense in the matrix extreme points of \textbf{$K$}.
\end{theorem}

\begin{proof}
	For each $n$, equip the space $M_n(V)$ with the operator norm induced by the norm on $V$ and equip the product space $M_n(\C) \times M_n(V)$ with the $L^2$ product norm
	$$
	\|(\mu, C)\| = \sqrt{\|\mu\|^2 + \|C\|^2}.
	$$
	Let $X \in K(n)$ be a matrix extreme point, $\gamma \in M_n(\C)$ an invertible matrix with tr$(\gamma^* \gamma)=1$, and $\epsilon >0.$ Then by \cite[Proposition 2.14]{HL}, the tuple $(\gamma^* \gamma, \gamma^* X\gamma)$ is a classical extreme point of $\Gamma_n(\textbf{$K$}).$ By \cite{K}, for any $\epsilon_1 >0$ (which we determine later) there is an exposed point $(\delta^* \delta, \delta^* Y\delta)$ of $\Gamma_n(\textbf{$K$})$ for some $r \in \mathbb{N},$ $Y\in K(r)$ and surjective $\delta \in M_{r,n}(\C)$ with tr$(\delta^*\delta)=1$ in the $\epsilon_1$-neighbourhood of $(\gamma^* \gamma, \gamma^* X\gamma).$
	By part (a) of Proposition \ref{tr19}, $Y$ is a matrix exposed point of \textbf{$K$}.
	
	We now argue that $r=n$ and $Y$ is of the same size as $X.$ Indeed, since $\gamma$ is invertible, so is $\gamma^* \gamma.$ Since $\delta^* \delta$ is $\epsilon_1$ close to $\gamma^* \gamma,$ we have that $\delta^* \delta$ is invertible as well (here we can take $\epsilon_1$ as small as needed). But since the rank of $\delta^* \delta$ is at most $r$ by assumption, we deduce that $r=n.$ 
	
	We proceed to prove that $X$ is close to a unitary conjugate $Y^\prime$ of $Y.$ As noted before, matrix exposed points are closed under conjugation by unitaries, hence $Y^\prime$ is also matrix exposed.  Firstly, by continuity of the functional calculus, for any $\epsilon_2 > 0,$ there is an $\epsilon_1 > 0$ (here we again make the above $\epsilon_1$ smaller if needed) such that if $\delta^* \delta$ is $\epsilon_1$ close to $\gamma^* \gamma,$ then $|\delta|$ is $\epsilon_2$ close to $|\gamma|.$ 
	Writing $\gamma = U_1 |\gamma|$ and $\delta = U_2 |\delta|$ in their polar decompositions with $U_1,U_2$ unitaries, we have that
	$$
	\|U_1^* \gamma - U_2^* \delta\| = \|\,|\gamma| - |\delta|\,\| < \epsilon_2,
	$$
	hence
	$$
	\|U_2 U_1^*\gamma - \delta\| = \| U_2 (U_1^* \gamma - U_2^* \delta) \| \leq \|U_2\| \|U_1^* \gamma - U_2^* \delta\| \leq \epsilon_2.
	$$
	Now denote $U_2U_1^*$ by $U$ and compute 
	\begin{align*}
		\| \delta^* Y\delta - \gamma^* U^*Y U\gamma \| \leq\  &\| \delta^* Y\delta - \delta^* YU \gamma\| + \|\delta^* YU \gamma - \gamma^* U^*Y U\gamma \|\\ \leq\ 
		& \|\delta^* Y \| \| \delta - U \gamma\| + \|(\delta - U\gamma)^* \| \|YU\gamma \|\\ \leq \ 
		& \|\delta^* Y \| \epsilon_2 + \epsilon_2 \|YU\gamma \|.
	\end{align*}
	Finally, 
	\begin{align*}
	\| X - U^* Y U \|  = & \|(\gamma^*)^{-1} (\gamma^* X \gamma - \gamma^* U^* Y U \gamma )\gamma^{-1} \|\\ \leq&\ 
	\|(\gamma^*)^{-1}\|\, \|\gamma^{-1}\| \, \|\gamma^* X \gamma - \gamma^* U^* Y U \gamma\|\\ \leq&\
	 \|(\gamma^*)^{-1}\|\, \|\gamma^{-1}\| \  \big(\|\gamma^* X \gamma -  \delta^* Y\delta\| + 
	\|\delta^* Y\delta -  \gamma^* U^* Y U \gamma\| \big)\\
	\leq&\ \|(\gamma^*)^{-1}\|\, \|\gamma^{-1}\| \ 
	(\epsilon_1 + \|\delta^* Y \| \epsilon_2 + \epsilon_2 \|YU\gamma \|),
	\end{align*}
	which can be made smaller than the given $\epsilon$ by a suitable choice of $\epsilon_1$ and $\epsilon_2.$
\end{proof}

We now present a spanning result for matrix exposed points. It extends the Straszewicz theorem \cite[Section II.2]{Ba}, more precisely Klee's generalization \cite{K}, to the matrix convex setting.

\begin{theorem}[\textbf{The Straszewicz-Klee theorem for matrix convex sets}]\label{th213}
	Let $K$ be a compact matrix convex set in a normed vector space $V.$ Then $\text{\textnormal{mexp}}(K) \neq \emptyset$ and
	$$
	K = \overline{\text{\textnormal{mconv}}}\,(\text{\textnormal{mexp}}(K)).
	$$
\end{theorem}

This theorem clearly follows from the fact that the matrix exposed points are dense in the matrix extreme points (Theorem \ref{th316}) and that the matrix extreme points span a compact matrix convex set (Theorem \ref{th: km}). An alternative proof goes along the lines of the presented proof of the matricial Krein-Milman theorem. In fact, Proposition \ref{tr19} can be used to reduce the matricial statement to the classical Klee theorem.
For more detail on the proof in the more general setting see \cite[Theorem 3.14]{KS22}, while we also note that \cite[Corollary 6.23]{Kr} uses a homogenization technique to tackle the finite-dimensional case.

\section{Free extreme points}

While matrix extreme points and matrix exposed points are natural extreme points for matrix convex sets, a shortcoming of both these types of extreme points is that there in fact may be too many of them. The classical Krein-Milman theorem not only guarantees that a compact convex set $C$ is the closed convex hull of its extreme points. Milman's converse also guarantees that if the closed convex hull of a set is $C$, then the closure of that set includes the extreme points. Thus it is of interest to determine the smallest class of extreme points that can recover a matrix convex set via matrix convex combinations. Before formally introducing the next family of extreme points, called free extreme points, let us illustrate the issue with an example. 

\begin{example}
\label{example:TooManyMatrixExtreme}
Let $X \in M_2 (\C^2)$ be the tuple
\[
X = \left(
\begin{bmatrix} 
1 & 0 \\
0 & -1
\end{bmatrix},\begin{bmatrix}
0 & 1 \\ 
1 & 0
\end{bmatrix}
\right).
\]
By setting $K:=\text{mconv} (X)$, we obtain that $K$ is a compact matrix convex set and that $K(1)$ is the closed unit disk in $\mathbb{R}^2$. Since the classical extreme points and classical exposed points of $K(1)$ are matrix extreme points and matrix exposed points of $K$, respectively, it follows that $K$ has infinitely many (nonunitarily equivalent) matrix extreme points and matrix exposed points. However, only a single element of $K$ is required to recover the set via matrix convex combinations.
\end{example}

This example highlights the possible redundancy of the matrix extreme points with respect to spanning properties. A matrix extreme point might be expressed as a compression of another matrix extreme point that lies at a higher level. A stronger extreme point condition was introduced by Kleski \cite{kls14} to address this issue, and points meeting this condition are now commonly called free extreme points. We note that free extreme points are closely tied to Arveson's notion of a boundary representation, which was developed earlier. This will be discussed in the following subsection.

\begin{definition}\label{fext}
	Let $K$ be a matrix convex set over $\FF^g$. A point $X \in K(n)$ is a \textbf{free extreme point} of $K$ if whenever $X$ is expressed as a matrix convex combination
 \[
 X = \sum_{i=1}^k \gamma_i^\ast Y^i \gamma_i
 \]
 where $\gamma_i \in M_{n_i,n}(\FF)$ are nonzero matrices with $\sum_{i=1}^k \gamma_i^\ast \gamma_i = I_n$, then for each $i$, either $n_i = n$ and $X \sim_u Y^i$ or $n_i > n$ and there exists a $Z^i \in K$ such that $X \oplus Z^i \sim_u Y^i$.  Denote the set of free extreme points of $K$ by fext$(K)$. 

\end{definition}

Intuitively, a point $X$ of $K$ is free extreme if it cannot be expressed as a nontrivial matrix convex combination of \textit{any} elements of $K$. In contrast, a point $X \in K(n)$ is matrix extreme if it cannot be expressed as a nontrivial matrix convex combination of elements of $K$ \textit{with size less than or equal to $n$}. The additional restrictions placed on free extreme points guarantee that they appear in every spanning set that consists of irreducible matrix tuples, as in the following theorem.

\begin{theorem}
    Let $K$ be a matrix convex set and let $E \subset K$ be a set of irreducible tuples that is closed under unitary conjugation. If $K = \textnormal{mconv}(E)$, then $\textnormal{fext}(K) \subset E$.
\end{theorem}
\begin{proof}
Suppose $X \in K \backslash E$ and $X \in \textnormal{mconv}(E)$. Then there exist $Y^i \in K(n_i)$ and $\gamma_i \in M_{n_i,n}(\C)$ such that
 \[
 X = \sum_{i=1}^k \gamma_i^\ast Y^i \gamma_i.
 \]
 Since $E$ is closed under unitary conjugation and $X \notin E$, it is not possible that $X \sim_u Y^i$ for any $i$. Furthermore, it is not possible that $X \oplus Z^i \sim_u Y^i$ for any $Z^i \in K$, as this would imply $Y^i$ is reducible and contradict the assumption that $E$ contains only irreducible tuples. It follows that $X$ cannot be a free extreme point of $K$.
\end{proof}

\begin{remark}
    In Example \ref{example:TooManyMatrixExtreme}, up to unitary equivalence, the matrix convex set $K$ has exactly one free extreme point, which is precisely the tuple $X$ that generates $K$. Thus free extreme points do give a minimal spanning set in that example. 
\end{remark}

So far, we have seen that the free extreme points are contained in every spanning set. However, not every compact matrix convex set is the closed matrix convex hull of its free extreme points. In fact, free extreme points may not exist at all. 

\begin{theorem}
    If $g \geq 1$, then there exist compact matrix convex sets in $M(\FF^g)$ that have no free extreme points. Similarly, if $g \geq 2$, then there exist compact matrix convex sets in $SM(\FF^g)$ that have no free extreme points. \end{theorem}

As it turns out, spanning results for free extreme points over a field are sensitive both to that field and to whether or not one works with self-adjoint variables; this must be taken into account when giving examples for the above theorem. See the upcoming Example \ref{example:Cuntz} for the use of the Cuntz isometries and row contractions, in the setting of complex matrix convex sets. However, this example no longer applies if one works with self-adjoint variables over the real coefficient field. In this case, \cite{Eve18} gives a construction of compact matrix convex sets that do not have free extreme points.  

While not every matrix convex set is spanned by its free extreme points, certain important families of matrix convex sets do admit a spanning result. Among these are matrix convex sets defined by real noncommutative polynomial inequalities, i.e., real free spectahedra, which are the matrix convex hull of their free extreme points. Before proving this result, we describe some tools from the operator system perspective on matrix convex sets.

\subsection{The operator system perspective}\label{subsec:opersys}

Matrix convex sets are dual to \textit{operator systems}. For brevity, we include the concrete presentation of an operator system only. See \cite[Chapter 13]{Pa} for the abstract point of view and its connection to the following definition.

\begin{definition}
An \textbf{operator system} is a vector subspace $\mathcal{S}$ of a unital $C^*$-algebra $A$ such that $1 \in \mathcal{S}$ and $\mathcal{S}$ is closed under the adjoint operation.
\end{definition}

The concept of an operator system is modeled after a key example in the commutative case. If $K$ is compact and convex inside a locally convex topological vector space, then the functions that respect the structure of $K$ are the continuous $\mathbb{R}$-affine functions $f \in A(K)$. The affine functions are closed under the adjoint, namely complex conjugation, as well as under vector space operations. However, unless $K$ consists of a single point, $A(K)$ is not closed under multiplication and therefore forms a proper subspace of $C(K).$ Thus, while integration against functions in $C(K)$ determines a regular Borel measure uniquely,
\[  \forall f \in C(K), \int f \, d\mu \,\, = \,\, \int f \, d\nu  \,\,\,\,  \hspace{.3 cm} \implies \hspace{.4 cm} \mu = \nu,\]
the same cannot be said for integration against only $f \in A(K)$.

\begin{definition} Let $z \in K$. Then a \textbf{representing measure} for $z$ is a regular Borel measure $\mu$ such that $\int f \,d\mu = f(z)$ for all $f \in A(K).$
\end{definition}

The study of integral representations underpins Choquet theory and the theory of boundaries. While we will not divulge all of the details here, the interested reader is pointed to \cite{Ch69} for more information; we will develop a few simple cases and state theorems as needed.

Consider the case when $z = t x + (1-t)y$ is not an extreme point of $K$. Then both the point mass $\delta_z$ and the nontrivial combination of point masses $\mu = t \delta_x + (1-t)\delta_y$ are representing measures for $z$. This is not a coincidence, as existence of a unique representing measure is equivalent to the claim that $z$ is extreme \cite{Bau}.

 This view may also be expanded into an analysis of extensions of maps. Unital positive linear functionals on $C(K)$ are exactly determined by integration against regular Borel probability measures, so a representing measure for $z \in K$ gives an example of a positive linear functional on $C(K)$ that, when restricted to $A(K)$, gives evaluation at $z$. Said differently, a representing measure for $z$ leads to a positive linear extension of the evaluation map $f \in A(K) \mapsto f(z)$ to the domain $C(K)$. Combining the above, we reach the following.

\begin{theorem}
Let $K$ be a compact convex subset of a locally convex topological vector space and let $z \in K$. Then $z$ is extreme precisely when the evaluation map $f \mapsto f(z)$ on $A(K)$ has a unique positive linear extension (which must be the evaluation map $f \mapsto f(z)$ on $C(K)$).
\end{theorem}

Note that evaluation $f \mapsto f(z)$ on $C(K)$ is not just a positive linear map; it is a representation, which is also trivially irreducible, as it is one-dimensional. Both of these ideas extend to the noncommutative setting. For operator systems, the appropriate morphism is a UCP map.

\begin{definition}
Let $\phi: \mathcal{S}_1 \to \mathcal{S}_2$ be a linear unital map between operator systems. Then $\phi$ is called \textbf{unital completely positive}, or \textbf{UCP}, if for every positive matrix $[s_{ij}]$ over $\mathcal{S}_1$, it follows that $[\phi(s_{ij})]$ is a positive matrix over $\mathcal{S}_2$.
\end{definition}

Roughly, when an operator system sits inside a noncommutative $C^*$-algebra, it is typically necessary to examine matrices over $\mathcal{S}$, as opposed to elements of $\mathcal{S}$, in order to see the full structure of $\mathcal{S}$. There is also a convention to use the word \lq\lq completely\rq\rq\hspace{0pt} to describe any property that also applies to matrices over $\cS$ of all sizes. For example, an isometric map preserves norms of elements of $\mathcal{S}$, whereas a completely isometric map preserves norms of all matrices (of any size) over $\mathcal{S}$.

It is particularly common to consider \lq\lq concrete\rq\rq\hspace{0pt} UCP maps $\phi: \mathcal{S} \to B(H)$. In this case, Arveson's Extension Theorem \cite[Theorem 1.2.3]{A1} states that $\phi$ extends to a UCP map on $C^*(\mathcal{S}),$ the $C^*$-algebra generated by $\cS$. Note that for one-dimensional Hilbert spaces, $B(H) \cong \mathbb{C}$, and every unital positive linear functional is automatically UCP by \cite[Proposition 1.2.2]{A1}. Arveson also developed the following notion of extreme point for operator systems, by analogy with the commutative case $A(K) \hookrightarrow C(K)$.

\begin{definition} \cite[Definition 2.1.1]{A1}
Let $\mathcal{S}$ be an operator system, sitting inside $A = C^*(\mathcal{S})$, and let $\pi: A \to B(H)$ be an irreducible representation. Then $\pi$ is called a \textbf{boundary representation} if $\pi$ is the unique UCP extension of $\pi|_\mathcal{S}$ to $A$.
\end{definition}

The boundary representations of the operator system $A(K)$ (which generates $C(K)$) are exactly the evaluation maps at extreme points. However, the noncommutative point of view also provides another useful perspective: dilation theory. 

\begin{definition}
Let $\phi: \mathcal{S} \to B(H)$ be a UCP map. Then a \textbf{UCP dilation} of $\phi$ is another UCP map $\psi: \mathcal{S} \to B(\widetilde{H})$ such that there exists an isometry $V: H \to \widetilde{H}$ with $\phi(s) = V^*\psi(s) V$ for all $s \in \cS$. That is,
\[ \psi \, \sim_u \, \begin{bmatrix} \phi & * \\ * & * \end{bmatrix}.\]
In particular, $\psi$ is called a \textbf{trivial UCP dilation} if it is unitarily equivalent to the direct sum of $\phi$ with another UCP map, and $\phi$ is called \textbf{maximal} if the only UCP dilations of $\phi$ are trivial.
\end{definition}

Maximal UCP maps turn out to have unique extensions, showing some utility of the noncommutative point of view.

\begin{theorem}\label{thm:maxext} \cite[Proposition 2.4]{Arv08}
Let $\phi: \mathcal{S} \to B(H)$ be a UCP map. Then $\phi$ is maximal if and only if both of the following hold: there is a unique UCP extension of $\phi$ to $C^*(\mathcal{S})$, and that extension is a representation of $C^*(\mathcal{S})$.
\end{theorem}

Consider, for example, the case of affine functions $A(K)$ embedded in continuous functions $C(K)$, and a point $z = \cfrac{x + y}{2}$ that is not extreme. Then the map

\[ f \in A(K) \,\, \mapsto \,\, \begin{bmatrix} f(x) & 0 \\ 0 & f(y) \end{bmatrix} \in M_2(\C) \]
is UCP, as is its unitary conjugation
\[ f \,\, \mapsto \,\, \begin{bmatrix} 1/\sqrt{2} & 1/\sqrt{2} \\ 1/\sqrt{2} & -1/\sqrt{2} \end{bmatrix} \begin{bmatrix} f(x) & 0 \\ 0 & f(y) \end{bmatrix} \begin{bmatrix} 1/\sqrt{2} & 1/\sqrt{2} \\ 1/\sqrt{2} & -1/\sqrt{2} \end{bmatrix} \]
\[ = \begin{bmatrix} \cfrac{f(x) + f(y)}{2} & \cfrac{f(x) - f(y)}{2} \\ \cfrac{f(x) - f(y)}{2} & \cfrac{f(x) + f(y)}{2} \end{bmatrix} = \begin{bmatrix} f(z) & \cfrac{f(x) - f(y)}{2} \\ \cfrac{f(x) - f(y)}{2} & f(z) \end{bmatrix}. \]
Here, the last equality is due to the fact that affine functions respect convex combinations, which yields $\frac{f(x) + f(y)}{2} = f\left( \frac{x + y}{2} \right) = f(z)$. So, there is a nontrivial UCP dilation of the evaluation map $f \mapsto f(z)$ on $A(K)$. This is consistent with the above theorem; evaluation at $z$ should not be maximal, as $z$ is not extreme, meaning evaluation $f \in A(K) \mapsto f(z)$ does not have a unique UCP extension.

When Arveson developed boundary representations as a notion of extreme point for operator systems, their sufficiency in the sense of a Krein-Milman theorem was not known in general, until the following result of Davidson and Kennedy.

\begin{theorem}\cite[\S 3]{DK15}
If $[s_{ij}]$ is a matrix over an operator system $\mathcal{S}$, then there is a boundary representation $\pi$ such that $\|[\pi(s_{ij})]\| = \|[s_{ij}]\|$. In particular, every operator system is completely normed by its boundary representations.
\end{theorem}

The interested reader should also see \cite{Arv08, DM05} for special cases, and \cite{F, kls14} for other key results about pure matrix states. The sufficiency of boundary representations provided a simple proof that there is a \lq\lq smallest\rq\rq\hspace{0pt} $C^*$-algebra, called the $C^*$-envelope, in which a given operator system $\mathcal{S}$ lives. The $C^*$-envelope can be found by taking the direct sum of boundary representations, but existence of the $C^*$-envelope had been shown by Hamana earlier in \cite{Ham79} using different techniques.

\begin{corollary}\label{cor:env}
If $\mathcal{S}$ is an operator system, then there is a $C^*$-algebra $C^*_e(\mathcal{S})$, called the \textbf{C*-envelope} of $\mathcal{S}$, which contains $\mathcal{S}$, is generated by $\mathcal{S}$, and has the following universal property. For any completely isometric embedding $\iota: \mathcal{S} \to A$ of $\mathcal{S}$ into a $C^*$-algebra $A$, there exists a unital $*$-homomorphism $\pi: C^*(\iota(\mathcal{S})) \to C^*_e(\mathcal{S})$ with $\pi \circ \iota$ acting as the identity on $\mathcal{S}$.
\end{corollary}

See \cite{Kak09} for more information about the $C^*$-envelope. In general, two copies of the same operator system $\mathcal{S}$ may generate $C^*$-algebras that look very different. However, if a copy of $\mathcal{S}$ sits inside a specific $C^*$-algebra, then there is a quotient of $C^*(\mathcal{S})$ that is isomorphic to the $C^*$-envelope. So, in this sense, the $C^*$-envelope of $\mathcal{S}$ is the \lq\lq smallest\rq\rq\hspace{0pt} $C^*$-algebra that contains $\mathcal{S}$. 

Our examples will primarily deal with operator systems spanned by a finite list of operators. For $T = (T_1, \ldots, T_d) \in B(H)^d$, we denote by $\mathcal{S}_T$ the smallest operator system containing every $T_i$. The operator system $\mathcal{S}_T$ is dual to the matrix convex set $\mathcal{W}(T)$, called the \textbf{matrix range} of $T$, defined as
\[ \mathcal{W}(T) = \bigcup\limits_{n=1}^\infty \mathcal{W}_n(T), \hspace{.5 cm} \mathcal{W}_n(T) = \{\phi(T) \,\, | \,\, \phi: \mathcal{S}_T \to M_n(\C) \text{ is UCP}\}.\]
In particular, knowledge of $\mathcal{W}(T)$ implies, through the use of limiting arguments, knowledge of every possible UCP image of $T$, and hence completely describes the operator system structure of $\mathcal{S}_T$.

\begin{example}
\label{example:Simplex}
Let $\Delta$ be a $2$-simplex given by the convex hull in $\mathbb{R}^2$ of $(0, 0)$, $(1, 0)$, and $(0, 1)$. Then by \cite[Theorem 4.1]{Pas}, there is a unique matrix convex set whose first level is $\Delta$. (Note that the portion of that theorem we actually use here is seen earlier in \cite[Theorem 4.7]{FNT}, just not in the language of matrix convex sets.) So, for any pair of operators $T = (T_1, T_2)$ such that $T_1$ and $T_2$ are positive, $T_1 + T_2 \preceq I$, and all three extreme points of $\Delta$ are in $\mathcal{W}_1(T)$, it follows that $\mathcal{S}_T$ will have the same operator system structure. However, the $C^*$-algebra generated by each concrete copy may look completely different.
\begin{itemize}
    \item The natural choice is to let $T$ be the direct sum of the extreme points of $\Delta$, so $C^*(\mathcal{S}_T)$ is commutative and in fact generates the $C^*$-envelope.
    \item Another choice is to let $T = (1, 0) \oplus (0, 1) \oplus (K_1, K_2)$ where $(K_1, K_2)$ is an irreducible pair of self-adjoint positive compact operators with $K_1 + K_2 \leq I$. In this case, $C^*(\mathcal{S}_T)$ is a unital $C^*$-algebra generated by compact operators, and it has an \lq\lq extra\rq\rq\hspace{0pt} representation that maps $(T_1, T_2) \mapsto (K_1, K_2)$. This example was used in \cite[Example 3.14]{Pas19} to show that minimal tuples for a certain matrix range need not be unique, even if the tuples consist of compact operators (see also the arXiv update to \cite[\S 6]{DDOSS17}).
    \item  Outside of the compact operators, the situation is even worse; \cite[Example 3.14]{PS19} describes how to construct an \textit{irreducible} tuple $T$, which must not consist of compact operators, whose matrix range is the unique matrix convex set over the simplex $\Delta$. This is in the same vein as Arveson's examples over the unit disk, as in \cite[p. 107]{A1}.
\end{itemize} 
This example also shows that there are subtleties when discussing the smallest \textit{spatial} presentation of an operator system \cite{DDOSS17, DP22, Pas19}; see \cite[Theorem 2.11 and Theorem 3.9]{DP22} for summary results. These issues do not appear in the setting of free spectrahedra, as in \cite[Theorem 1.2]{HKM13} and \cite[Theorem 1.2]{Zal17}.
\end{example}

 When $A$ is a tuple of matrices, as opposed to a tuple of operators, $\mathcal{W}(A)$ agrees with the matrix convex hull (with or without a closure) of $A$, which is a consequence of Choi's Theorem \cite{Choi} and compactness arguments.

\begin{proposition}
If $A \in M_n(\C)^d$, then $\mathcal{W}(A) = \operatorname{mconv}(A) = \overline{\operatorname{mconv}}(A)$.
\end{proposition}

When restricted to the setting of irreducible UCP maps, Theorem \ref{thm:maxext} equates boundary representations $\pi$ (or more precisely, the restrictions $\pi|_\mathcal{S}$ of boundary representations) with maximal, irreducible UCP maps on $\mathcal{S}$. Therefore, as seen in \cite[Theorem 1.1 (3)]{EHKM} and \cite[Theorem 4.2]{kls14} with different nomenclature, free extreme points line up with \textit{finite-dimensional} boundary representations.

\begin{corollary}
\label{cor:TrivDilations}
Let $T \in B(H)^d$, and consider the matrix range $\mathcal{W}(T)$. Then the free extreme points of $\mathcal{W}(T)$ are precisely the images $\pi(T)$ of boundary representations $\pi$ of $\mathcal{S}_T$ that map into matrix algebras of arbitrary finite dimension. Equivalently, for any matrix convex set $K \subseteq M(\C)^d$, an element $A \in K$ is a free extreme point of $K$ if and only if $A$ is irreducible and every dilation $D = \begin{bmatrix} A & * \\ * & * \end{bmatrix} \in K$ is a trivial dilation $D = A \oplus B$.
\end{corollary}

\begin{remark} 
To test that $A \in K(n)$ is a free extreme point of $K$, it is only necessary to check dilations of matrix dimension $n + 1$. In particular, any nontrivial dilation of $A$ will compress to a nontrivial dilation of $A$ that has dimension $n + 1$.
\end{remark}

Consequently, the spanning problem for free extreme points is equivalent to determining if the set of \textit{finite-dimensional} boundary representations is sufficient, as opposed to the set of \textit{all} boundary representations. This distinction is a crucial one; while boundary representations are the natural extreme points from an operator algebraist's point of view, they are not always appropriate when dealing with the setting of free spectrahedra, which we consider in the next section. Said differently, a finite-dimensional problem suggests a finite-dimensional solution.

Whether one works in the finite-dimensional setting or not, the relationship between unique UCP extensions and maximal UCP maps is quite powerful. Consider, for example, when an operator system is spanned by unitaries, as a unitary is precisely a contraction operator that cannot be nontrivially dilated to another contraction. This is particularly useful when dealing with operator systems inside group $C^*$-algebras, as in \cite{FKPT14, Kav14, KPTT13}. Further, as a consequence of \cite{DK15} (see also \cite{Arv08,DM05}), if one extends the notion of a compact matrix convex set to include ``infinite-dimensional" levels as in \cite{DK}, then the resulting notion of maximal irreducible element makes no distinction between finite or infinite dimension. Such ``operator convex" sets are thus the operator convex hull of their, perhaps infinite-dimensional, maximal irreducible elements. For the types of problems we consider, it is natural to ask if what starts in finite dimensions stays in finite dimension. Because of this finite-dimensional focus, we will not expand on the operator convex point of view. 

\section{Free spectrahedra and free convex semialgebraic geometry}\label{sec:freespecsec}

A natural class of matrix convex sets to consider consists of sets defined by noncommutative polynomial inequalities. These sets, called free spectrahedra, arise in settings such as quantum information and games \cite{BHK,BN,BJN,DDN,PR} and linear systems control \cite{HKM12,HPMV09,dOHMP09}. Furthermore, they enable tractable relaxations of problems that are NP-hard in the classical setting \cite{HKM13}. 

\subsection{Linear matrix inequalities and free spectrahedra}

Given a $g$-tuple $A \in SM_m(\FF)^g$, define the monic linear pencil $L_A(x)$ by
\[
L_A (x):= I_m-A_1 x_1 -\dots -A_g x_g.
\]
The inequality $L_A(x) \succeq 0$ is called a \textbf{linear matrix inequality}. Classically, the set of $x \in \RR^g$ such that $L_A(x) \succeq 0$ is a convex set called a spectrahedron. These spectrahedra are the feasibility domains of semidefinite programs and, as a consequence, arise in a number of applications, e.g., see \cite{BPT12,BEFB94,PS03}. 

The dimension-free generalization of a spectrahedron is obtained by allowing the tuple $X$ to be a $g$-tuple of self-adjoint matrices that enter the linear matrix inequality via the Kronecker (tensor) product. More precisely, for $X \in SM_n (\FF)^g$, define the (free) monic linear pencil $L_A(X)$ by
\[
L_A (X):= I_m \otimes I_n - A_1 \otimes X_1 - \dots - A_g \otimes X_g,
\]
and also let 
\[ \Lambda_A (X) := A_1 \otimes X_1 + \dots + A_g \otimes X_g\] denote the linear part of $L_A (X)$ for future reference. Define the matrix convex set $\cD_A$, called a \textbf{free spectrahedron}, by 
\[
\cD_A := \bigcup_{n = 1}^\infty \cD_A (n), \hspace{.8 cm} \cD_A (n) := \{X \in SM_n(\FF)^g \ : \ L_A (X) \succeq 0 \}.
\]

The extreme points of free spectrahedra are particularly well understood. In fact, as shown by the following theorem of \cite{EEHK22}, determining if an element of a free spectrahedron is an extreme point is equivalent to solving a linear system. See also \cite{EHKM} and \cite{Kr} for related results. 

\begin{theorem}\cite[Theorem 2.6]{EEHK22} \label{theorem:Equations}
Let $\cD_A \subset SM(\FF)^g $ be a bounded free spectrahedron and $X \in \cD_A(n)$. Additionally let $K_{A,X}$ be a matrix whose columns form an orthonormal basis for the kernel of $L_A(X)$. 
\begin{enumerate}[\rm (1)]
    \item \label{it:ArvEq}
    $X$ is a free extreme point of $\cD_A$ if and only if $X$ is irreducible and the only solution to the homogeneous linear equations
    \begin{equation}
    \label{eq:ArvEq}
         (A_1 \otimes \beta_1^* + \dots + A_n\otimes\beta_g^*) K_{A,X} = 0
    \end{equation}
    in the unknown $(\beta_1, \ldots, \beta_g) \in M_{n,1}(\FF)^g$ is $(\beta_1, \ldots, \beta_g) = 0$.
    \item \label{it:EucEq}
    $X$ is a classical extreme point of $\cD_A(n)$ if and only if the only solution to the homogeneous linear equations
    \[
       (A_1 \otimes \beta_1 + \dots + A_n\otimes\beta_g) K_{A,X} = 0
    \]
    in the unknown $(\beta_1, \ldots, \beta_g) \in SM_n(\FF)^g$ is $(\beta_1, \ldots, \beta_g) = 0$.
    \item \label{it:MatEq}
    $X$ is a matrix extreme point of $\cD_A$ if and only if the only solution to the homogeneous linear equations
    \begin{align*}
        (I\otimes\beta_0 + A_1 \otimes \beta_1 + \dots + A_n\otimes\beta_g) K_{A,X} = 0 \\
        \mathrm{tr}(\beta_0 + X_1^* \beta_1 + \dots + X_g^* \beta_{g}) = 0
    \end{align*}
    in the unknown $(\beta_0,\beta_1,\dots,\beta_g) \in SM_{n}(\FF)^{g+1}$ is $(\beta_0,\beta_1,\dots,\beta_g) = 0$.
\end{enumerate}
\end{theorem}

\begin{proof}
The proofs for all three cases are similar, so we restrict our attention to item \ref{it:ArvEq}. Using Corollary \ref{cor:TrivDilations}, the tuple $X \in \cD_A(n)$ is a free extreme point of $\cD_A$ if and only if $X$ is irreducible and the dilations of $X$ of dimension $n + 1$ satisfy
\[
L_A \left(\begin{bmatrix}
    X & \beta \\
    \beta^* & \psi 
\end{bmatrix} \right) \succeq 0 \,\,\,\,\,\, \implies \,\,\,\,\,\, \beta = 0.
\]
Conjugating by the canonical shuffle, see \cite[Chapter 8]{Pa}, shows that
\[
L_A \left(\begin{bmatrix}
    X & \beta \\
    \beta^* & \psi 
\end{bmatrix} \right) \,\,\,\, \sim_u \,\,\,\,
\begin{bmatrix}
    L_A (X) & -\Lambda_A (\beta) \\
    -\Lambda_A (\beta^*) & L_A (\psi)
\end{bmatrix}.
\]
Taking the Schur complement shows that
\[
\begin{bmatrix}
    L_A (X) & -\Lambda_A (\beta) \\
    -\Lambda_A (\beta^*) & L_A (\psi)
\end{bmatrix} \succeq 0
\]
if and only if $L_A (\psi) \succeq 0$ and
\[
L_A (X) - \Lambda_A (\beta) L_A(\psi)^\dagger \Lambda_A (\beta^*) \succeq 0.
\]
Here $L_A(\psi)^\dagger$ denotes the Moore-Penrose pseudo inverse of $L_A (\psi)$. If the above equation holds, then it is straightforward to check that the kernel of $\Lambda_A(\beta)^*$ must contain the kernel of $L_A(X)$. That is, equation \eqref{eq:ArvEq} must hold. On the other hand, if \eqref{eq:ArvEq} holds, then by choosing $\alpha > 0$ sufficiently small, one has
\[
L_A (X) - \Lambda_A (\alpha \beta) \Lambda_A (\alpha \beta^*) = L_A (X) - \alpha^2 \Lambda_A (\beta) \Lambda_A (\beta^*),
\]
in which case 
\[
L_A \left(\begin{bmatrix}
    X & \alpha \beta \\
    \alpha \beta^* & 0
\end{bmatrix} \right) \succeq 0.
\]
In short, finding tuples $\beta$ that give valid dilations of $X$ is equivalent to solving equation \eqref{eq:ArvEq}.\looseness=-1
\end{proof}

It is important to note that in each of the three cases of Theorem \ref{theorem:Equations}, the tuple $\beta$ is chosen from a different domain. For example, case \ref{it:ArvEq} concerns a $g$-tuple of column vectors, whereas case \ref{it:EucEq} concerns a $g$-tuple of self-adjoint square matrices. However, when $n = 1$ and $\mathbb{F} = \mathbb{R}$, those two domains are the same. So, in addition to being computationally useful, Theorem \ref{theorem:Equations} can be used to guarantee the existence of free extreme points at level one of bounded free spectrahedra that are closed with respect to complex conjugation. 


\begin{corollary}\cite[Proposition 6.1]{EHKM}
\label{cor:SpecLevel1Extreme}
    Let $\cD_A \subset SM(\FF)^g$ be a free spectrahedron and assume $\cD_A(2) = \overline{\cD_A(2)}$. Then $x \in \cD_A(1)$ is a classical extreme point if and only if it is a free extreme point. 
\end{corollary}
\begin{proof}
    We will consider only the case that $\cD_A$ is bounded. If $\FF = \RR$, the conditions for a tuple to be a free or classical extreme point of $\cD_A(1)$ in Theorem \ref{theorem:Equations} above are identical. The proof over $\CC$ when the second level of $\cD_A$ is closed under complex conjugation can be obtained by considering real and imaginary parts of solutions to \eqref{eq:ArvEq} above.
\end{proof}

\subsection{A spanning result for free spectrahedra}

Bounded free spectrahedra that are closed under complex conjugation form a large class of matrix convex sets that are known to be the matrix convex hull of their free extreme points. Furthermore, for these sets there is a Caratheodory-like bound on the sum of sizes of the extreme points required.

\begin{theorem}\label{thm:dropspan}\cite[Theorem 1.3]{EH}
    Let $\cD_A \subset SM(\FF)^g$ be a bounded free spectrahedron that is closed under complex conjugation and let $X \in \cD_A(n)$. Then $X$ can be written as a matrix convex combination
    \[
    X= \sum_{i=1}^k \gamma_i^* F^i \gamma_i \qquad \mathrm{s.t.} \qquad \sum_{i=1}^k \gamma_i^* \gamma_i = I
    \]
    of free extreme points $F^i \in \cD_A(n_i)$. Furthermore, if $\FF=\RR$, then the extreme points can be chosen so that $\sum_{i=1}^k n_i \leq n(g+1)$. Similarly, if $\FF=\CC$, then $\sum_{i=1}^k n_i \leq 2n(g+1)$. These expressions can be computed via semidefinite programming. 
\end{theorem}

The above theorem was later extended in \cite{Eve23} to include generalizations of free spectrahedra that are defined by linear pencils with compact operator coefficients. An algorithm for computing expansions of elements of a free spectrahedron in terms of free extreme points is also available in \cite{EEdO+}. To simplify the argument for Theorem \ref{thm:dropspan}, we focus on the proof when $\FF = \RR$. The proof when $\FF = \CC$ and $\cD_A $ is closed under complex conjugation is accomplished by reducing to the real case. This reduction is originally done for free spectrahedra in \cite[Section 3]{EH}. Later, \cite[Theorem 2.4]{Eve23} proved that for any matrix convex set $K \subset SM(\C)^g$ that is closed under complex conjugation, free extreme points span $K$ over the complex numbers if and only if free extreme points span $K \cap SM(\RR)^g$ over the real numbers. 

Inspired by works such as \cite{Arv08,DK15,DM05} that give spanning results in the infinite-dimensional setting, a key idea in the proof of this theorem is the notion of a maximal 1-dilation, which can be used to ``exhaust the set of valid dilations" of a given tuple. 

Let $\cD_A \subset SM(\RR)^g$ be a bounded real free spectrahedron and let $X \in \cD_A$. Define the \textbf{dilation subspace} of $\cD_A$ at $X$, denoted $\mathfrak{K}_{A,X}$, by
\[
\mathfrak{K}_{A,X} := \{\beta \in M_{n,1} (\RR)^g \ : \ker L_A (X) \subseteq \ker \Lambda_A(\beta^*)\}.
\]
As a consequence of Theorem \ref{theorem:Equations}, a tuple $Y \in \cD_A$ is (up to unitary equivalence) a direct sum of free extreme points of $\cD_A$ if and only if $\dim \mathfrak{K}_{A,Y} = 0$. 

To ensure that this dilation process remains finite-dimensional, \cite{EH} introduced maximal 1-dilations, which increase the size of the current tuple by exactly one. The goal is to ensure that these one step dilations always decrease the dimension of the dilation subspace. If this is achieved, then since the dimension of the dilation subspace of any element of $\cD_A(n)$ is bounded above by $ng$, it will follow that the tuple $X$ can be dilated to a direct sum of free extreme points using a sequence of at most $ng$ maximal 1-dilations. The desired expression of $X$ as a matrix convex combination of those free extreme points is then straightforward to obtain.

\begin{definition}
    \label{def:max1diDef} 
Given a bounded real free spectrahedron $\cD_A \subset SM(\RR)^g$ and a tuple $X \in \cD_A (n)$, say the dilation 
\[
\hat{Y}=\begin{bmatrix}
X & \hbeta \\
\hbeta^* & \hpsi
\end{bmatrix} \in \cD_A (n+1)
\]
is a \textbf{maximal $1$-dilation} of $X$ if $\hbeta \in M_{n,1}(\RR)^g$ is nonzero and the following two conditions hold:
\begin{enumerate}
\item \label{it:1isMaximizerCompact} The real number $1$ satisfies
\[
\begin{array}{rllcl}  1 =&\underset{\alpha \in \RR, \psi \in \RR^g}{\max} \ \ \ \ \alpha \\
\mathrm{s.t.} & L_A \left(\begin{bmatrix}
X & \alpha \hbeta \\
\alpha \hbeta^* & \psi
\end{bmatrix}\right) \succeq 0.
\end{array}
\]
\item \label{it:hgamIsExtremeCompact} $\hpsi$ is an extreme point of the closed bounded convex set 
\begin{equation}
\gengam{X}{\hbeta}:=\label{eq:gammaset}
\left\{
\psi \in \RR^g \ : \ \begin{bmatrix}
    X  & \hbeta \\
    \hbeta^* & \psi
\end{bmatrix} \in \cD_A
\right\}.
\end{equation}
\end{enumerate}

\end{definition}

\begin{remark}
    For a fixed $X \in \cD_A$ and $\hbeta \in \mathfrak{K}_{A,X}$ that satisfies item \eqref{it:1isMaximizerCompact} above, the set defined in equation \eqref{eq:gammaset} is in fact a nonempty compact spectrahedron, which can be seen by applying the Schur complement. 
\end{remark}

\begin{theorem}\cite[Theorem 2.4]{EH}
Let $\cD_A \subset SM(\RR)^g$ be a bounded real free spectrahedron and $X \in \cD_A(n)$. Assume that $X$ is not unitarily equivalent to a direct sum of free extreme points of $\cD_A$. Then there exists a tuple $\hY \in \cD_A(n+1)$ that is a maximal 1-dilation of $X$. Furthermore, any such $\hY$ satisfies
\[
\dim \mathfrak{K}_{A,\hY} < \dim \mathfrak{K}_{A,X}.
\]
\end{theorem}

\begin{proof}

First note that the existence of maximal $1$-dilations follows from a routine compactness argument. Now, let $\hY$ be a maximal $1$-dilation of $X$ and assume towards a contradiction that 
\[
\dim \mathfrak{K}_{A,\hY} \geq \dim \mathfrak{K}_{A,X}.
\]
Let $\eta \in M_{n,1} (\RR)^g$ and $\sigma' \in \mathbb{R}^g$ be tuples such that
\[
\begin{bmatrix}
    \eta \\
    \sigma' 
\end{bmatrix} \in \mathfrak{K}_{A,\hY}.
\]
By considering the Schur complement as in the proof of Theorem \ref{theorem:Equations}, it follows that there exists a constant $\alpha>0$ such that
\[
\begin{bmatrix}
X & \hat{\beta} & \alpha \eta \\
\hat{\beta}^* &\hpsi & c\sigma' \\
\alpha\eta^* & \alpha\sigma' & 0
\end{bmatrix} \in \cD_A.
\]
Using the matrix convexity of $\cD_A$ it follows that
\[
\begin{bmatrix} 1 & 0 & 0 \\
0 & 0 & 1
\end{bmatrix}
\begin{bmatrix}
X & \hat{\beta} & \alpha \eta \\
\hat{\beta}^* &\hpsi & c\sigma' \\
\alpha\eta^* & \alpha\sigma' & 0
\end{bmatrix} 
\begin{bmatrix} 1 & 0 \\
0 & 0 \\
0 & 1 \end{bmatrix} = \begin{bmatrix}
    X & \alpha \eta \\
    \alpha \eta & 0
\end{bmatrix} \in \cD_A.
\]
From this, we obtain that $\eta \in \mathfrak{K}_{A,X}$. Using this, one can argue that the assumption
\[
\dim \mathfrak{K}_{A,\hY} \geq \dim \mathfrak{K}_{A,X}
\]
implies that there exists a real number $c$ and a tuple $\sigma \in \RR^g$ so that
\begin{equation}
\label{eq:badBetaDiCompact}
L_A \left(\begin{bmatrix}
X & \hat{\beta} & c \hat{\beta} \\
\hat{\beta}^* & \hpsi & \sigma \\
c\hat{\beta}^* & \sigma & 0
\end{bmatrix}\right) \succeq 0
\end{equation}
and such that either $ c \neq 0$ or $\sigma \neq 0$. 

Now, applying the NC LDL$^*$-decomposition (see \cite[Appendix]{EH}) shows that, up to unitary equivalence, inequality \eqref{eq:badBetaDiCompact} holds if and only if $L_A (X) \succeq 0$ and the Schur complements
\begin{equation}
\label{eq:3LDLCompactMiddleCompact}
I_d-c^2 Q \succeq 0
\end{equation}
and
\begin{equation}
\label{eq:3LDLCompact}
L_A (\hpsi)-Q-\left(\Lambda_A(\sigma)-cQ\right)^*\left(I_d-c^2 Q\right)^\dagger \left(\Lambda_A(\sigma)-cQ\right) \succeq 0,
\end{equation}
where 
\begin{equation}
\label{eq:QdefCompact}
Q:= \Lambda_A (\hbeta^*) L_A (X)^\dagger \Lambda_A (\hbeta).
\end{equation}
It follows that
\begin{equation}
\label{eq:gammaSchurCompact}
L_A (\hpsi)-Q \succeq 0
\end{equation}
and
\begin{equation}
\label{eq:3diaKerContainCompact}
\ker  [L_A (\hpsi)-Q] \subseteq \ker [\Lambda_A(\sigma)-c Q].
\end{equation}

Picking $\tilde{\alpha} > 0$ so that $\tilde{\alpha}\|\Lambda_A(\sigma)-c Q\|$ is smaller than the smallest nonzero eigenvalue of $L_A (\hpsi)-Q$ and using inequalities \eqref{eq:gammaSchurCompact} and \eqref{eq:3diaKerContainCompact} guarantees
\[
L_A (\hpsi)-Q \  \pm \ \tilde{\alpha} \left(\Lambda_A(\sigma)-c Q\right) \succeq 0.
\]
It follows from the above that
\begin{equation}
\label{eq:betterBetaSchurCompact}
\begin{array}{rllcl}
&L_A (\hpsi \pm \tilde{\alpha} \sigma)- (1 \pm c\tilde{\alpha} )Q\\
=& L_A (\hpsi \pm \alpha \sigma)-\left(\Lambda_A (\sqrt{1 \pm c\tilde{\alpha} }\hat{\beta}^*)L_A(X)^\dagger\Lambda_A(\sqrt{1 \pm c\tilde{\alpha} }\hat{\beta})\right)& \succeq & 0.
\end{array}
\end{equation}
Since $L_A(X) \succeq 0$, equation \eqref{eq:betterBetaSchurCompact} implies (by checking the Schur complement) that
\begin{equation}
\label{eq:betterBetaDiCompact}
L_A \left(\begin{bmatrix}
X & \sqrt{1 \pm c\tilde{\alpha} }\hat{\beta} \\
\sqrt{1 \pm c\tilde{\alpha} }\hat{\beta}^* & \hpsi\pm \tilde{\alpha}  \sigma
\end{bmatrix}\right) \succeq 0.
\end{equation}
Recalling that $\hY$ is a maximal $1$-dilation of $X$, we must have 
\[
\sqrt{1 \pm c\tilde{\alpha} }  \leq 1.
\]
This implies that $c\tilde{\alpha} =0$. Since $\tilde{\alpha} > 0$, it follows that $c = 0$. From our construction, this in turn implies that $\sigma \neq 0$. But then equation \eqref{eq:betterBetaDiCompact} implies that 
\[
\hpsi\pm \tilde{\alpha}\sigma \in \gengam{X}{\hbeta},
\] 
which contradicts the fact that $\hpsi$ is an extreme point of the convex set $\gengam{X}{\hbeta}$. We conclude that
$\dim \mathfrak{K}_{A,\hY} < \dim \mathfrak{K}_{A,X},$
as claimed.
\end{proof}

The desired spanning result now follows easily: if a point $X$ is not already a direct sum of free extreme points, then it may be successively dilated in order to reduce the dimension of the dilation subspace. Once this dimension reaches zero, the final dilation is the direct sum of free extreme points, and $X$ is in the matrix convex hull of those summands.

\subsection{Contrasts with the classical setting}

Any noncommutative polynomial with matrix coefficients, not just a linear pencil, can receive matrix inputs via the tensor product. However, linear matrix inequalities have the same power in expressing matrix convex sets as general polynomial inequalities, in stark contrast with the classical case (see \cite{Bra11}).

\begin{theorem}\cite[Theorem 1.4]{HM12}
\label{th:HM12}
    Let $p$ be a noncommutative polynomial in $g$ self-adjoint variables such that $p(0) \succ 0$. Define $\cD_p \subset SM(\FF)^g$ to be the principal component of the set
    \[
    \{X \in SM(\FF)^g : p(X) \succeq 0\}.
    \]
    Assume that $\cD_p(n)$ is convex for each $n$. Then $\cD_p$ is matrix convex, and there exists a monic linear pencil $L_A$ such that $\cD_p = \cD_A.$ That is, $\cD_p$ is a free spectrahedron.
\end{theorem}

We emphasize that there are two surprising aspects of the above theorem. First, every matrix convex set defined by a noncommutative polynomial inequality can be defined by a linear matrix inequality. Second, for dimension-free sets defined by noncommutative polynomials, levelwise convexity is sufficient for matrix convexity.  See also \cite[Theorem 4.2]{Kr} for further discussion. 

From Theorem \ref{th:HM12}, a free semialgebraic matrix convex set is a free spectrahedron. Classically, the Tarski principle states that a projection of a semialgebraic set is again semialgebraic. Surprisingly, this corner-stone of real semialgebraic geometry does not extend to the free setting. To show this, we first formally introduce free spectrahedrops, i.e., projections of free spectrahedra. 

Let $A \in SM_m (\FF)^g$ and $B \in SM_m(\FF)^h$ be tuples of self-adjoint matrices. Given $(X,Y) \in SM_n (\FF)^{g+h}$ where $X=(X_1,\dots,X_g)$ and $Y=(Y_1,\dots,Y_h)$, define
\[
L_{(A,B)} (X,Y) := I_m \otimes I_n - A_1 \otimes X_1 - \dots -A_g \otimes X_g - B_1 \otimes Y_1 - \dots - B_h \otimes Y_h.
\]
The free spectrahedrop $\dropAB$ is the coordinate projection onto the $X$ variables in the free spectrahedron $\cD_{(A,B)}$. That is,
\[
\dropAB := \{ X \in SM(\FF)^g \ : \,\exists \, Y \in SM(\FF)^h \ \mathrm{s.t.} \ L_{(A,B)}(X,Y) \succeq 0 \}. 
\]
It is straightforward to verify that free spectrahedrops are matrix convex. However, not every free spectrahedrop is a free spectrahedron.

\begin{theorem}\cite[Proposition 9.8]{HM12}
There exists a free spectrahedrop that is not a free spectrahedron. 
\end{theorem}
\begin{proof}
The original proof of this result in \cite{HM12} first argues that any matrix convex set with level one equal to
\[
\mathcal{C}=\{(x_1,x_2): 1-x_1^4-x_2^4 \geq 0\}
\]
cannot be a free spectrahedron. In fact, $\mathcal{C}$ itself is not a spectrahedron since it fails the line test \cite{HV}. On the other hand, \cite{HM12} provides an explicit construction of a free spectrahedrop $K$ such that $K(1)=\mathcal{C}$.

 An alternative proof that highlights the differences between free extreme points of free spectrahedra and free spectrahedrops can be obtained by considering the free polar dual of a free spectrahedron. Given a matrix convex set $K \subseteq SM(\FF)^g$, the free polar dual of $K$, denoted $K^\circ,$ is defined by 
\[
K^\circ:= \{Y \in SM(\FF)^g \ : \ L_X (Y) \succeq 0 \mathrm{ \ for \ all \ } X \in K\}.
\]
Take $A$ to be an irreducible tuple in $SM_m (\RR)^g$ for $m \geq 2$ such that $\cD_A$ is a bounded free spectrahedron. Using \cite[Theorem 1.2]{EHKM}, the free polar dual of a bounded free spectrahedron $\cD_A$ is precisely mconv$(A)$. Furthermore, \cite[Theorem 4.11]{HKM17} shows that the polar dual of a bounded free spectrahedrop is a bounded free spectrahedrop, hence mconv$(A)$ is a bounded free spectrahedrop. Since $A$ is irreducible, it is straightforward to show that up to unitary equivalence, the only free extreme point of mconv$(A)$ is $A$, hence mconv$(A)$ does not have free extreme points at level $1$. On the other hand, mconv$(A)$ is closed under complex conjugation since $A$ has real entries. Using Corollary \ref{cor:SpecLevel1Extreme} we see that mconv$(A)$ cannot be a free spectrahedron. 
\end{proof}

We have seen from Theorem \ref{thm:dropspan} that real free spectrahedra are the matrix convex hulls of their free extreme points. This implies that the corresponding operator systems are completely normed by their finite-dimensional boundary representations. However, infinite-dimensional boundary representations of those operator systems may exist. See \cite[Proposition 7.1]{EHKM}, \cite[Proposition 5.5]{FKPT14}, and \cite[Proposition 7.9 and Example 7.10]{Kr} for information about the matrix convex set in the following example. Also, see \cite[Corollary 10.13 and Remark 11.6]{Kav14}, \cite[Corollary 3.9 and Corollary 3.11]{PP21}, and \cite[Theorem 6.7]{Pas} for how this and similar examples fit into the context of approximation theorems.

\begin{example} \cite[\S 5]{FKPT14}
Consider the set of tuples of self-adjoint contractions, also known as the free cube: 
\[ \mathcal{C}_g = \{X \in SM(\mathbb{C})^g: \,\,\, \forall j, \, \, -I \preceq X_j \preceq I\}. \]
This is a bounded free spectrahedron that is also closed under entrywise complex conjugation. Therefore, $\mathcal{C}_g$ is the matrix convex hull of its free extreme points by Theorem \ref{thm:dropspan}. The free extreme points of $\mathcal{C}_g$ are precisely the tuples of self-adjoint unitaries, which are also the classical extreme points. A more direct dilation addresses the spanning problem in this case: dilate each $X_j$ to the self-adjoint unitary 
\[\begin{bmatrix} X_j & \sqrt{I - X_j^2} \\ \sqrt{I - X_j^2} & -X_j \end{bmatrix}\] 
as in \cite{Hal50}. The corresponding operator system is spanned by universal self-adjoint unitaries, and therefore it sits inside the full group $C^*$-algebra of $(\mathbb{Z}/2\mathbb{Z})^{*g}$, where $*$ denotes the free product of groups. In dimension $g = 2$, all irreducible pairs of self-adjoint unitaries are of matrix dimension $1$ or $2$. However, for $g \geq 3$, the free product is nonamenable and there are irreducible tuples of self-adjoint unitaries of any dimension, including infinite dimension. In particular, for $g \geq 3$ there are infinite-dimensional boundary representations of the operator system, even though the finite-dimensional ones (i.e., the free extreme points) are sufficient.
\end{example}

Finally, the spanning result Theorem \ref{thm:dropspan} does not extend to all complex free spectrahedra, particularly those that are not closed under complex conjugation.

\begin{example}\label{example:Cuntz} \cite[Example 6.30]{Kr}
Let $\mathcal{O}_d$ be the Cuntz algebra \cite{Cun77}, $d \geq 2$. This is the $C^*$-algebra generated by bounded operators $\mathscr{T}_1, \mathscr{T}_2, \ldots, \mathscr{T}_d$ on a separable Hilbert space such that each $\mathscr{T}_i$ is an isometry, that is,
\[ \forall i \in \{1, \ldots, d\}, \,\, \mathscr{T}_i^*\mathscr{T}_i = I,\]
and the additional condition
\[ \mathscr{T}_1\mathscr{T}_1^* + \mathscr{T}_2\mathscr{T}_2^* + \ldots + \mathscr{T}_d \mathscr{T}_d^* = I\]
holds. In particular, these constraints imply that the ranges of the isometries $\mathscr{T}_i$ are pairwise orthogonal. If $\mathscr{T} = (\mathscr{T}_1, \ldots, \mathscr{T}_d)$, then $\mathcal{W}(\mathscr{T})$ has no free extreme points. Indeed, every free extreme point $A$ of $\mathcal{W}(\mathscr{T})$ would be the image of $\mathscr{T}$ under a (finite-dimensional) boundary representation, and no representations $\mathcal{O}_d \to M_n(\C)$ exist. One way to see this is that every matrix isometry is automatically a unitary, but in fact, a much stronger claim is true: the Cuntz algebra is a simple $C^*$-algebra (of infinite dimension), so all of its representations are injective. Note also that because the $C^*$-envelope of $\mathcal{S}_T$ is realized as a quotient of $C^*(\mathcal{S}_T) = \mathcal{O}_d$, but $\mathcal{O}_d$ is simple, it follows that $\mathcal{O}_d$ itself is the $C^*$-envelope!
\end{example}

As in \cite[Theorem 3.8]{Zh14}, the matrix range of the Cuntz isometries above is simply the set of row contractions (a consequence of the dilation results in \cite{Pop89}), which is the set of matrix tuples $(T_1, \ldots, T_d)$ satisfying $\sum\limits_{j=1}^d T_j T_j^* \preceq I$.
When the matrices $T_j$ are put into self-adjoint coordinates $T_j = X_j + iY_j$, the corresponding set of self-adjoint tuples $(X_1, Y_1, \ldots, X_d, Y_d)$ with $\sum\limits_{j=1}^d (X_j + iY_j)(X_j - iY_j) \preceq I$ is a bounded free spectrahedron. In fact, it is a spectraball, as in \cite[\S 1]{HKMV20}. However, that free spectrahedron in $2d$ self-adjoint variables is not closed under simultaneous entrywise complex conjugation 
\[ X_j \mapsto \overline{X_j}, \hspace{.6 cm} Y_j \mapsto \overline{Y_j}.\]
For instance, as in the text after \cite[Remark 2.5]{Pas22}, the row contraction $(E_{12}, E_{22})$ splits into the self-adjoint tuple
\[ (X_1, Y_1, X_2, Y_2) = \left( \begin{bmatrix} 0 & 1/2 \\ 1/2 & 0 \end{bmatrix}, \begin{bmatrix} 0 & -i/2 \\ i/2 & 0 \end{bmatrix}, \begin{bmatrix} 0 & 0 \\ 0 & 1 \end{bmatrix}, \begin{bmatrix} 0 & 0 \\ 0 & 0 \end{bmatrix}   \right)\]
with 
\[ (X_1 + iY_1)(X_1 - iY_1) + (X_2 + iY_2)(X_2 - iY_2) \,\, = \,\, E_{12}E_{12}^* + E_{22}E_{22}^* \,\, = \,\, I,\] 
but the tuple $(\overline{X_1}, \overline{Y_1}, \overline{X_2}, \overline{Y_2})$ of complex conjugates no longer meets the condition needed:
\[ (\overline{X_1} + i\overline{Y_1})(\overline{X_1} - i\overline{Y_1}) + (\overline{X_2} + i\overline{Y_2})(\overline{X_2} - i\overline{Y_2}) \,\, = \,\, E_{21}E_{21}^* + E_{22}E_{22}^* \,\, = \,\, 2 E_{22} \,\, \npreceq I. \]
So, this example does not fall under Theorem \ref{thm:dropspan}. In contrast, the self-adjoint matrix ball 
\[ \mathcal{B}_g = \{X \in SM(\C)^g: \, X_1^2 + \ldots + X_g^2 \preceq I\}\]
is the set of self-adjoint row contractions, and this \textit{is} closed under entrywise complex conjugation. So, the self-adjoint matrix ball \textit{does} fall under Theorem \ref{thm:dropspan}. Combining these two extremes provides examples of a different spanning phenomenon.

\begin{definition}\cite[Definition 2.6]{Pas22}
Let $d \geq 1$ and $g \geq 0$. Define $\mathcal{M}_{d, g}$ as the set of tuples $(T_1, \ldots, T_d, X_1, \ldots, X_g)$ of square complex matrices such that each $X_j$ is self-adjoint and $\sum\limits_{i=1}^d T_iT_i^* + \sum\limits_{j=1}^g X_j^2 \preceq I$.
\end{definition}

The matrix convex set $\mathcal{M}_{d,g}$ can be shown to be a free spectrahedron (after it is placed into self-adjoint coordinates, meaning there are $2d+g$ self-adjoint variables) by a similar computation as for the row contractions or matrix ball; see the discussion after \cite[Definition 2.6]{Pas22} for the linear matrix inequality formulation. Further, as a closed and bounded matrix convex set, $\mathcal{M}_{d,g}$ is also the matrix range of some operator tuple $(\mathcal{T}, \mathcal{X})$. For example, one may take the direct sum of matrices chosen from a dense subset of each level. In this case, a tuple of \textit{operators} (not necessarily matrices) $(T, X)$ is the UCP image of $(\mathcal{T}, \mathcal{X})$ if and only if
\begin{equation}\label{eq:rowpsa} X_j = X_j^* \,\,\,\, \text{ and } \,\,\,\, \sum\limits_{i=1}^d T_i T_i^* + \sum\limits_{j=1}^g X_j^2 \preceq I.\end{equation}
In general, the choice of \lq\lq universal\rq\rq\hspace{0pt} tuple $(\mathcal{T}, \mathcal{X})$ is immaterial when discussing only the operator system or matrix convex set itself. However, since the case $d \geq 2$, $g = 0$ returns to the setting of row contractions, in that case the most natural choice is a tuple of Cuntz isometries. In general, $\mathcal{M}_{d, g}$ for $g \geq 1$ will not have a simple $C^*$-envelope, as there are too many boundary representations.

\begin{theorem}\label{theorem:m1gmaximal} \cite[Theorem 2.8]{Pas22}
Let $d \geq 1$ and $g \geq 0$, and let $(T, X)$ be a tuple of operators satisfying $(\ref{eq:rowpsa})$. Then $(T, X)$ admits no nontrivial dilations that satisfy $(\ref{eq:rowpsa})$ if and only if each $T_i$ is injective, the ranges of the $T_i$ are linearly independent subspaces (trivially satisfied if $d = 1$), and $\sum\limits_{i=1}^d T_iT_i^* + \sum\limits_{j=1}^g X_j^2 = I$.
\end{theorem}

If $\mathcal{W}(\mathcal{T}, \mathcal{X}) = \mathcal{M}_{d, g}$, then the images of $(\mathcal{T}, \mathcal{X})$ under boundary representations are precisely the irreducible tuples that satisfy (\ref{eq:rowpsa}) and admit no nontrivial dilations that satisfy (\ref{eq:rowpsa}). So, the theorem determines all boundary representations of the corresponding operator system. In previously considered cases (e.g. $d \geq 2$ and $g = 0$ for Cuntz isometries, or $d = 1$ and $g = 0$ for unitaries), the theorem recovers key examples. However, it also leads to a distinct spanning result.

\begin{corollary} \cite[Corollary 2.10]{Pas22} 
Fix $d = 1$ and $g \geq 1$. Then the free spectrahedron $\mathcal{M}_{1,g}$ is the closed matrix convex hull, but not the matrix convex hull, of its free extreme points. In particular, any tuple in $\mathcal{M}_{1,g}(n)$ is the limit of a sequence of matrix convex combinations of free extreme points, where each combination is of the form $\sum\limits_{i=1}^{m} \gamma_i^* F^i \gamma_i$ with $\sum\limits_{i=1}^m \operatorname{dim}(F^i) \leq 2n$.
\end{corollary}

\begin{proof}
Note that because every element of $\mathcal{M}_{1,g}$ is a row contraction (with additional constraints), the first level of $\mathcal{M}_{1,g}$ is the Euclidean ball of real dimension $2d+g = 2+g$. In particular, it includes the point $w = (0, \ldots, 0, 1)$. Now, $w$ is not free extreme; it fails the condition of the previous theorem since $0$ does not define an injective operator. However, $w$ is a classical extreme point of the ball. Hence, if we assume $w$ is the matrix convex combination of free extreme points, then since the corresponding realization
\[ w = \sum_{i=1}^n \gamma_i^* F^i \gamma_i \]
where the $F^i$ are free extreme just gives a convex combination of vector states, we conclude $w$ is the compression of a single free extreme point. So, $w$ admits a finite-dimensional dilation $F = (T, X_1, \ldots, X_g)$ where $T$ is injective and $TT^* + \sum\limits_{j=1}^g X_j^2 = I$. That is,
\[ T \, \sim_u \, \begin{bmatrix} 0 & a \\ b & W \end{bmatrix}, \,\,\,\,\,\,\, X_j \, \sim_u \, \begin{bmatrix} 0 & \beta_j \\ \beta_j^* & P_j \end{bmatrix} \, \forall j \in \{1, \ldots, g - 1\}, \,\,\,\,\,\,\, X_g \, \sim_u \, \begin{bmatrix} 1 & \beta_g \\ \beta_g^* & P_g \end{bmatrix}. \]
The top left entry of $I = TT^* + \sum\limits_{j=1}^g X_j^2$ is $1 = aa^* + \sum\limits_{j=1}^g \beta_j\beta_j^* + 1$, so $a = 0$ and each $\beta_j = 0$. We conclude that $T$ has a row of zeroes. This is a contradiction, as $T$ is an injective operator on a finite-dimensional space, and hence it is an invertible matrix. It follows that $\mathcal{M}_{1,g}$ is not the matrix convex hull of its free extreme points.

Now, consider the closed matrix convex hull of free extreme points instead. If we consider an arbitrary element $(T, X_1, \ldots, X_g)$ of $\mathcal{M}_{1,g}$, then we may approximate the tuple in order to assume $T$ is an invertible matrix and $TT^* + \sum\limits_{j=1}^g X_j^2 \preceq (1 - \varepsilon)I$ for some $\varepsilon > 0$. Fix $A \not= 0$ such that $TT^* + \sum\limits_{j=1}^g X_j^2 = I - AA^*$. We will construct a maximal dilation of the form
\[ S = \begin{bmatrix} T & A \\ B & C \end{bmatrix}, \,\,\,  Y_1 = \begin{bmatrix} X_1 & 0 \\ 0 & D \end{bmatrix}, \,\,\, Y_j = \begin{bmatrix} X_j & 0 \\ 0 & 0 \end{bmatrix} \,\,\, \forall j \in \{2, \ldots, g\}.\] 
First let $C = \delta I$, with $\delta > 0$ chosen small enough that $B := -CA^*(T^{-1})^*$ satisfies $BB^* + CC^* \preceq I$. By design, $TB^* + AC^* = 0 = BT^* + CA^*$, so the matrix $S$ has $SS^* = \begin{bmatrix} TT^* + AA^* & 0 \\ 0 & BB^* + CC^* \end{bmatrix}$. Since both $T$ and $C$ are invertible matrices, $SS^*$ dominates a positive multiple of the identity, so $S$ is also an invertible matrix. Finally, if we set $D := \sqrt{I - BB^* - CC^*}$, then
\[ SS^* + \sum\limits_{j=1}^g Y_j^2 \, = \, \begin{bmatrix} TT^* + AA^* + \sum\limits_{j=1}^g Y_j^2 & 0 \\ 0 & BB^* + CC^* + (I - BB^* - CC^*) \end{bmatrix} \, = \, \begin{bmatrix} I & 0 \\ 0 & I \end{bmatrix}.\]
The matrix tuple $(S, Y_1, \ldots, Y_g)$ meets all the conditions of Theorem \ref{theorem:m1gmaximal}, so it admits no nontrivial dilations in $\mathcal{M}_{1,g}$, and hence it decomposes as a direct sum of free extreme points. Since every point of $\mathcal{M}_{1,g}$ could be approximated to arbitrary precision by compressions of a direct sum of free extreme points, the proof of the closed spanning result is complete. Moreover, since the matrix dimension of the dilation $(S, Y_1, \ldots, Y_g)$ was twice the matrix dimension of $(T, X_1, \ldots, X_g)$, the final claim of the corollary also follows.
\end{proof}

\section*{Acknowledgments}

The authors are grateful both to the referee and to Orr Shalit for helpful commentary.

\section{Summary}
In the dimension-free setting, matrix convex sets have many families of extreme points, which carry with them both analogues of classical theorems and stark departures from the classical case. Matrix extreme points admit a spanning theorem, but they are overly redundant, unlike extreme points in the classical setting. However, just as exposed points are dense in the extreme points, so too are matrix exposed points dense in the matrix extreme points. Boundary representations give rise to a type of extreme point that completely norms any operator system (the dual of a matrix convex set), but these objects are often infinite-dimensional. In the finite-dimensional setting, the analogous type of extreme point is called a free extreme point. Whenever free extreme points span a matrix convex set, they in fact form a minimal spanning set, but free extreme points might not exist at all. This is a major barrier for problems where finite-dimensionality of the extreme points is crucial. 

For bounded free spectrahedra, which are matrix convex sets determined by a linear matrix inequality, the sufficiency of free extreme points comes down to the coefficient field. For bounded real free spectrahedra, or bounded complex free spectrahedra that are closed under complex conjugation, free extreme points form a minimal spanning set. Moreover, free extreme points in this setting can be detected using the kernel of a certain finite-dimensional operator, which lends itself to a robust dilation technique in finitely many steps. However, this spanning theorem does not apply in the general complex case, as seen by operator-algebraic constructions.

 \Addresses

\end{document}